\documentclass[11pt]{article}%
\usepackage{amssymb,amsmath,amsfonts,amsthm,array,bm,bbm,color}%
\providecommand{\U}[1]{\protect\rule{.1in}{.1in}}

\numberwithin{equation}{section}

\newtheorem{theorem}{Theorem}[section]
\newtheorem{lemma}[theorem]{Lemma}

\newtheorem{proposition}[theorem]{Proposition}
\newtheorem{remark}[theorem]{Remark}

\newtheorem{definition}[theorem]{Definition}

\def\<{\langle}
\def\>{\rangle}
\def\d{{\rm d}}

\def\div{{\rm div}}
\def\E{\mathbb{E}}
\def\N{\mathbb{N}}
\def\P{\mathbb{P}}
\def\R{\mathbb{R}}
\def\T{\mathbb{T}}
\def\Z{\mathbb{Z}}
\def\F{\mathcal{F}}

\newcommand{\assign}{:=}

\begin{document}

\title{Dissipation enhancement by transport noise for\\ stochastic $p$-Laplace equations}

\author{Zhao Dong\quad Dejun Luo\quad Bin Tang \bigskip \\
{\footnotesize Key Laboratory of RCSDS, Academy of Mathematics and Systems Science,}\\
{\footnotesize Chinese Academy of Sciences, Beijing 100190, China} \\
{\footnotesize School of Mathematical Sciences, University of Chinese Academy of Sciences, Beijing 100049, China}}

\maketitle

\vspace{-20pt}

\begin{abstract}
The stochastic $p$-Laplace equation with multiplicative transport noise is studied on the torus $\T^d\, (d\geq 2)$. It is shown that the dissipation is enhanced by transport noise in both the averaged sense and the pathwise sense.
\end{abstract}

\textbf{Keywords:} $p$-Laplace, transport noise, dissipation enhancement, semigroup approach

\textbf{MSC (2020):} 60H15, 60H50

\section{Introduction}

We are concerned with the stochastic $p$-Laplace equation on the torus  $\T^d= \R^d/ \Z^d\, (d\geq 2)$, perturbed by multiplicative noise of transport type:
  \begin{equation}\label{stoch-p-laplace}
  d u = \Delta_p u \,d t + \nabla u \circ d W, \quad u(0,\cdot) = u_0,
  \end{equation}
where $\Delta_p u= \div(|\nabla u|^{p-2} \nabla u)$ with $p>2$, $\circ\, dW$ means the stochastic differential is understood in the Stratonovich sense and $W= W(t,x)$ is a space-time noise, white in time and coloured in space, modelling some background random perturbation. We shall also assume that $W(t,x)$ is divergence free in the space variable; see Section \ref{sec: noise setting} for its precise form. At least formally, one has the energy balance: $\P$-a.s. for all $t\geq 0$,
\begin{equation}\label{energy identity}
    \|u(t)\|_{L^2}^2 + 2 \int_0^t \|\nabla u(s)\|_{L^p}^p\,d s  = \|u_0 \|_{L^2}^2;
\end{equation}
similarly as in the non-perturbed case, this immediately implies decay of solutions:
  \begin{equation}\label{deterministic-dissip}
  \|u(t) \|_{L^2} \leq  \frac{\|u_0 \|_{L^2} }{\big(1+ (p-2)\lambda_1^{p/2} t \|u_0 \|_{L^2}^{p-2} \big)^{1/(p-2)}},
  \end{equation}
where $\lambda_1$ is the spectral gap of $\T^d$.
Inspired by the recent work \cite{FHX21} in the deterministic setting, we will show that suitably chosen random noise  can greatly enhance the rate of decay, both in averaged sense and in the pathwise sense. Before stating the precise results, let us briefly recall some literature related to the phenomena of dissipation enhancement.

It has been known for a long time, in the physics and engineering communities, that certain perturbations speed up the mixing of fluids. In the mathematical literature, early influential studies of such phenomena date back to the works \cite{BHN05, CKRZ08}. In particular, for a bounded divergence free vector field $v$ on a smooth bounded domain $D\subset \R^N$, the authors of \cite{BHN05} studied the eigenvalue problem for the operators $-\Delta + A v \cdot \nabla\ (A>0)$ with Dirichlet boundary condition; it was shown that the principal eigenvalue $\lambda_A$ remains bounded as $A\to +\infty$ if and only if $v$ admits a first integral $w\in H^1_0(D)$, namely, $v\cdot \nabla w=0$. Later on, Constantin et al. \cite{CKRZ08} extended such ideas to the setting of compact manifolds $M$, and proved that an incompressible flow $v$ on $M$ fulfils the relaxation-enhancing property if and only if $v\cdot \nabla$ has no nontrivial eigenfunction in $H^1(M)$. Recall that an incompressible flow $v$ on $M$ is called relaxation enhancing if  the solution $\phi^A$ to the advection-diffusion equation
  $$\partial_t \phi - \Delta \phi + A v\cdot \nabla \phi =0 ,\quad \phi(0,\cdot) =\phi_0 \in L^2(M)  $$
satisfies, for any $t>0$,
  $$\lim_{A\to \infty} \|\phi^A(t,\cdot)- \bar\phi_0 \|_{L^2(M)}= 0, $$
where $\bar\phi_0 = \frac1{|M|}\int_M \phi_0 \,\d x$ is the average of $\phi_0$, a quantity preserved by the above equation.

Motivated by the pioneering works \cite{BHN05, CKRZ08}, the phenomena of dissipation enhancement have been studied intensively in the past years, see for instance \cite{Zla10, BCZ17, FI19, CZDE20, BBPS21, GY21, IXZ21} among many others. In particular, Feng and Iyer \cite{FI19} gave some explicit estimates on the dissipation time of $v$ in terms of the so-called mixing rate; the former means the minimal time needed for the solutions of advection-diffusion equations to dissipate a fraction (e.g. $1/2$) of their $L^2$-norms. Later on, Iyer et al. \cite{IXZ21} have used the dissipation time to formulate conditions on incompressible flows so that the blow-up in certain nonlinear systems is suppressed by the perturbation of such flows. In the remarkable work \cite{BBPS21}, Bedrossian et al. have shown that the incompressible flows $v$ can be chosen as the sample paths of certain stationary solutions of stochastic 2D Navier-Stokes equations, showing the generality of flows with dissipation-enhancing properties. Their arguments are quite technical, relying on the quantitative Harris theorem from ergodic theory and delicate studies of the underlying projective process. More recently, Gess and Yaroslavtsev \cite{GY21} obtained some related results for the famous Kraichnan model from turbulence theory.

On the other hand, there is a series of studies on stochastic fluid dynamical equations with transport noise; the latter might be regarded as random analogues of the vector field $v$ appearing above. The vorticity form of stochastic 2D Euler equation driven by transport noise has been studied for various types of initial data, see e.g. \cite{BFM16, BM19, FL19, FL20, LC22}. In particular,  it was shown in \cite{FL20} that white noise solutions of a sequence of stochastic 2D Euler equations converge weakly to the unique stationary solution of the 2D Navier-Stokes equation driven by space-time white noise. A remarkable fact is that the approximating equations are formally conservative, while the limit equation is dissipative. Partly inspired by this work, Galeati \cite{Galeati20} established a scaling limit result for $L^2$-solutions of stochastic linear transport equations on $\T^d$ with transport noise:
  $$d u + b\cdot \nabla u\,d t + \sqrt{\nu}\, \nabla u\circ d W=0, $$
where $b$ is a time-dependent vector field with appropriate regularity and $\nu$ is the intensity of noise. Rescaling the noise in a suitable way, he showed that the solutions converge to the unique solution of the deterministic parabolic equation
  $$\partial_t u + b\cdot \nabla u = \nu \Delta u.$$
Loosely speaking, small scale transport noise produces in the limit an extra dissipative term, which can be called eddy dissipation. Such scaling limit result was later on extended in \cite{FGL21a, Luo21} to some stochastic fluid equations with transport noise, see \cite{FGL21c} for quantitative convergence rates. Note that the bigger the noise intensity, the stronger the dissipation term in the limit equation; we have made use of this fact to show that transport noise suppresses possible explosion of solutions to some deterministic equations, cf. \cite{FL21, FGL21b}.

Moreover, using mild formulations of both approximating stochastic equations and the limit equation, we have shown in \cite{FGL21c} dissipation enhancement for stochastic heat equations with transport noise, i.e. \eqref{stoch-p-laplace} with $p=2$. In this paper, we aim to extend such result to the case $p>2$ corresponding to nonlinear diffusions. The same problem has already been studied in the deterministic setting by Feng, Hu and Xu in \cite{FHX21}, to which we also refer for discussions on the motivation of studying $p$-Laplace evolution systems. Following some ideas in \cite{CZDE20, FI19},  they introduced a nonlinear version of dissipation time $\kappa_d$ (see \cite[Definition 2.3]{FHX21}), and estimated the latter via the mixing rate of the time-dependent incompressible flow $v$, see \cite[Theorems 2.5 and 2.7]{FHX21} for general estimates and Corollaries 2.6 and 2.8 therein for more explicit results when $v$ is assumed to be strongly/weakly mixing.

Our first  result gives enhanced dissipation in the average sense.
\begin{theorem}\label{thm-average}
Let $p\in (2, 2d/(d-2))$ and $R>0$ be given. For any $\mu>0$, there exists a space-time noise $W(t,x)$ such that for any $\| u_0\|_{L^2}\leq R $, the solution $\{u(t) \}_{t>0}$ to \eqref{stoch-p-laplace} with initial condition $u_0$ satisfies
\begin{equation}\label{thm-average.eq}
    \|u(t) \|_{L^2(\Omega,L^2)} \leq \frac{\|u_0 \|_{L^2} }{\big(1+ (p-2)\mu t \|u_0 \|_{L^2}^{p-2} \big)^{1/(p-2)}}, \quad \mbox{for any } t\geq 1.
\end{equation}
\end{theorem}

Taking $\mu$ big enough, we see that the above estimate improves \eqref{deterministic-dissip}, though in a weaker averaged sense.  We can also prove a pathwise assertion on dissipation enhancement.

\begin{theorem}\label{thm-pathwise}
Let $p\in (2, 2d/(d-2))$ and $R>0$ be given. For any $\mu>0$ and $0<q< \frac{\ln 2}{\mu(p-2)^2}$, there exists a noise $W(t,x)$ such that for any $\| u_0\|_{L^2}\leq R,\, u_0\neq 0 $, there is a random constant $C(\omega)>0$ with finite $q$-th moment such that the solution  $\{u(t) \}_{t>0}$ to \eqref{stoch-p-laplace} satisfies, $\P$-a.s. for all $t>0$,
\begin{equation}\label{thm-pathwise.eq}
    \exp\bigg(- \frac1{\|u(t) \|_{L^2}^{p-2}} \bigg) \leq C(\omega) e^{-(p-2)\mu t} \exp\bigg(- \frac1{\|u_0 \|_{L^2}^{p-2}} \bigg);
\end{equation}
if $t>\frac{\ln C(\omega)}{(p-2)\mu}$, the above inequality can be rewritten as
\[\|u(t) \|_{L^2} \leq \frac{\|u_0 \|_{L^2}}{\big(1+\left((p-2)\mu t-\ln C(\omega) \right) \|u_0 \|_{L^2}^{p-2}\, \big)^{1/(p-2)}}. \]
\end{theorem}

To prove Theorem \ref{thm-pathwise}, we will use the Borel-Cantelli lemma as in the proof of \cite[Theorem 1.9]{FGL21c} (see also \cite[Section 7]{BBPS22}), and apply the dissipation estimates obtained in the proof of Theorem \ref{thm-average} on the intervals $[n-1,n],\, n\geq 1$. Here, the main new idea is that we choose noises with some fixed intensity $\kappa$ but different coefficients $\{\theta_{k,n} \}_{k\in \Z^d_0}$ on these intervals; as $n$ increases, the noises produce stronger and stronger dissipation enhancement in the averaged sense. By Chebyshev's inequality, this implies the probabilities of a sequence of events are summable, and then Borel-Cantelli's lemma gives us dissipation enhancement in the pathwise sense, see Section \ref{sec:dissipation.2} for details.

We also mention that, for fixed $p>2$, the range of $q$ provided by Theorem \ref{thm-pathwise} becomes very small for big $\mu>0$. In fact, we can enlarge the range by replacing the exponential $2^n$ in the proof with $a^n$ for some big $a>2$; then we can get an upper bound like $\frac{\ln a}{\mu(p-2)^2}$. We omit the details in this work.

We finish the short introduction with the organization of the paper. In Section \ref{sec: Preliminaries}, we introduce some notations for functional setting and recall some useful results. We present in Section \ref{sec: noise setting} the exact choice of the space-time noise $W(t,x)$, for which the stochastic $p$-Laplace equation \eqref{stoch-p-laplace} admits variational solutions; we also show that the solution can be rewritten in mild formulation. Using the semigroup approach as in \cite{FGL21c}, Theorems \ref{thm-average} and \ref{thm-pathwise} will be proved in Sections \ref{sec:dissipation.1} and \ref{sec:dissipation.2}, respectively. Finally, we demonstrate in the appendix the well posedness of \eqref{stoch-p-laplace} using the variational framework.

\section{Preliminaries}\label{sec: Preliminaries}


In this section, we introduce some notations and provide several important tools that will be used in the paper.

Let $\T^d=\R^d/\Z^d$ be the $d$-dimensional torus, $d \geq 2$; $\Z_0^d = \Z^d \setminus \{0\}$ is the set of nonzero integer points. The brackets $\langle \cdot,  \cdot \rangle$ stand for the inner product in $L^2(\T^d)$ and the duality between elements in $W^{1, p}(\T^d)$ and $W^{-1, \frac{p}{p-1}}(\T^d)$.
The norms in spaces $L^2(\T^d)$, $L^p(\T^d)$, and $W^{1,p}(\T^d)$ are denoted as $\|\cdot\|_{L^2}$, $\|\cdot\|_{L^p}$, and $\|\cdot \|_{W^{1, p}}$, respectively. Let $\{e_k \}_k$ be the usual complex basis of $L^2(\T^d,  \mathbb C)$.
As the $p$-Laplace equations considered below preserve the means of solutions,  we assume in this paper that the function spaces consist of functions on $\T^d$ with zero average.

Let $\Delta$ be the Laplace operator on $\T^d$ and $\lambda_k =4 \pi^2 |k|^2, \, k \in \Z^d,$ are the eigenvalues of $\Delta$. We denote by $\{ e^{t \Delta} \}_{t \geq 0}$ the usual heat semigroup on $\T^d$. For $p >2,$ we define the $p$-Laplace operator as below:
\begin{equation}\label{def,p-laplace}
    \Delta_p u \assign \div (|\nabla u|^{p - 2} \nabla u).
\end{equation}
We recall some properties of $\Delta_p$.

\begin{proposition}\label{pro.p-laplace}
    The p-Laplace operator $\Delta_p $ satisfies the estimates below:
    \begin{equation*}
        \begin{split}
            \langle \Delta_p u - \Delta_p v,  u - v \rangle \leq 0 , &\\
            \langle \Delta_p v,  v \rangle \leq - c \| v \|_{W^{1,p}}^p, &
        \end{split}
    \end{equation*}
    for all $u,v \in W^{1,p}(\T^d),$ where $c$ is a positive number that only depends on $p$. In particular, the second assertion implies
      $$\|\Delta_p v \|_{W^{-1, \frac{p}{p-1}}} \leq c \| v \|_{W^{1,p}}^{p-1}. $$
\end{proposition}

\begin{proof}
    The proof can be found in \cite[Example 4.1.9]{liu15}.
\end{proof}


The next lemma (see \cite[Lemma 4.2]{FHX21}) is an easy fact, but is necessary to deal with Case 1 of Section \ref{sec:dissipation.1}.

\begin{lemma}\label{P1}
    For any $p>2$, it holds
   \[ 1 - \frac{2 x}{p - 2} \leq \frac{1}{(1 + x)^{\frac2{p-2}}} , \quad  x > 0.\]
\end{lemma}

\begin{proof}
    The linear function $1 - \frac{2 x}{p - 2}$ and the convex function  $\frac{1}{(1 + x)^{\frac2{p-2}}}$
    are tangent at $x=0$. The inequality holds by the property of convex functions.
\end{proof}

The following lemma is motivated by \cite[Lemma 4.5]{FHX21}; it is the key ingredient for running the iteration argument when proving Theorem \ref{thm-average} and Proposition \ref{dissipation.2}.

\begin{lemma}\label{P2}
    For $p>2,$ suppose $a,b \geq 0$, and
    \begin{equation*}
        y \leq \frac{x}{\big(1+a x^{\frac{p-2}{2}} \big)^{\frac{2}{p-2}}}, \quad z \leq \frac{y}{\big(1+b y^{\frac{p-2}{2}}\big)^{\frac{2}{p-2}}} ,\quad x,y,z \geq 0,
    \end{equation*}
    then we have
    \[z \leq \frac{x}{\big(1+(a+b) x^{\frac{p-2}{2}} \big)^{\frac{2}{p-2}}} .\]
\end{lemma}

\begin{proof}
    Note that $\frac{y}{\big(1+b y^{\frac{p-2}{2}} \big)^{\frac{2}{p-2}}}$ is monotonically increasing in the variable $y$, it holds
    \begin{equation*}
        \begin{split}
            z &\leq \frac{y}{\big(1+b y^{\frac{p-2}{2}} \big)^{\frac{2}{p-2}}} \leq \frac{x}{\big(1+a x^{\frac{p-2}{2}} \big)^{\frac{2}{p-2}}} \cdot \frac{1}{\left(1+b \frac{x^{\frac{p-2}{2}}}{1+a x^{\frac{p-2}{2}}}\right)^{\frac{2}{p-2}}} \\
            & \leq \frac{x}{\big(1+(a+b) x^{\frac{p-2}{2}} \big)^{\frac{2}{p-2}}} .
        \end{split}
    \end{equation*}
\end{proof}

We state the classical heat semigroup estimates (see e.g. \cite{Lun95}) for later use.
\begin{lemma}
    \label{P3} For any $ 2 \leq p \leq + \infty, $ and any $T>0$, there exists a
    constant $C>0,$ such that
    \[ \|e^{ t \Delta} \phi \|_{W^{1,  p}} \leq C t^{-\big( \frac{1}{2} + \frac{d}{2} ( \frac{1}{2} - \frac{1}{p}) \big)} \| \phi \|_{L^2} \]
    for any $\phi \in L^{2}(\mathbb{T}^d), $ and any $t \in (0, T].$
\end{lemma}

\section{Choice of noise $W$ and mild formulation of \eqref{stoch-p-laplace}}\label{sec: noise setting}

In Section \ref{subsec:noise setting}, we describe the choice of the noise $W(t, x)$ and show the existence and uniqueness of variational solutions of SPDE \eqref{stoch-p-laplace}, by using a general result proved in the appendix.
Section \ref{subsec:mild form} explains the link between variational solution and mild form of SPDE \eqref{stoch-p-laplace}.

\subsection{Choice of noise}\label{subsec:noise setting}

As in \cite[Section 1.3]{FGL21c}, we perturb the equations \eqref{stoch-p-laplace} with the space-time noise below:
\begin{equation}\label{noise}
    W(t, x)= \sqrt{C_{d} \kappa}\,  \sum_{k\in \Z^d_0}\sum_{i=1}^{d-1} \theta_{k} \sigma_{k, i}(x) W^{k, i}_{t} ,
\end{equation}
where $C_{d}=d/(d-1)$ is a normalizing constant,  $\kappa>0$ is the noise intensity and $\theta\in\ell^{2} =\ell^{2}(\Z^d_0)$,
the space of square summable sequences indexed by $\Z_0^d$. $\{W^{k, i}:k\in\mathbb{Z}^{d}_{0},  i=1, \ldots, d-1\}$
are standard complex Brownian motions defined on a filtered probability space $(\Omega,  \mathcal F,  (\mathcal F_t),  \P)$, satisfying
\begin{equation}\label{noise.1}
    \overline{W^{k, i}} = W^{-k, i},  \quad\big[W^{k, i}, W^{l, j} \big]_{t}= 2t \delta_{k, -l} \delta_{i, j} .
\end{equation}
$\{\sigma_{k, i}: k\in\mathbb{Z}^{d}_{0},  i=1, \ldots, d-1\}$ are divergence free vector fields on $\T^d$ defined as
\begin{equation}\label{sigma}
    \sigma_{k, i}(x) = a_{k, i} e_{k}(x) ,
\end{equation}
where $\{a_{k, i}\}_{k, i}$ is a subset of the unit sphere $\mathbb{S}^{d-1}$ such that: (i) $a_{k, i}=a_{-k, i}$ for all
$k\in \mathbb{Z}^{d}_{0}, \,  i=1, \ldots, d-1$; (ii) for fixed $k$,  $\{a_{k, i}\}_{i=1}^{d-1}$ is an ONB of
$k^{\perp}=\{y\in\mathbb{R}^{d}:y\cdot k=0 \}$. It holds that $\sigma_{k, i}\cdot \nabla e_k = \sigma_{k, i}\cdot \nabla e_{-k} \equiv 0$
for all $k\in \Z^d_0$ and $1\leq i\leq d-1$.

We shall always assume that
\begin{itemize}
    \item $\theta \in \ell^2$ is symmetric,  i.e. $\theta_k = \theta_l$ for all $k, l\in \Z^d_0$ satisfying $|k|=|l|$;
    \item $\|\theta \|_{\ell^2} = \big(\sum_{k \in \Z_0^d} \theta_k^2 \big)^{1/2} =1$.
\end{itemize}
In fact, under the assumptions above, the noise is $W$ real. We rewrite \eqref{noise} as
\begin{equation}\label{noise.2}
    \begin{split}
        W(t, x) &= \sqrt{C_{d} \kappa}\,  \sum_{k\in \Z^d_0}\sum_{i=1}^{d-1}
        \theta_{k} \left\{ {\rm Re}(\sigma_{k, i}(x)) {\rm Re}(W^{k, i}_{t})-{\rm Im}(\sigma_{k, i}(x)) {\rm Im}(W^{k, i}_{t}) \right\} \\
        & = \sum_{k\in \Z^d_0}\sum_{i=1}^{d-1} \xi_{k, i}(x) B_t^{k, i} .
    \end{split}
\end{equation}
$\{\xi_{k, i}: k\in\mathbb{Z}^{d}_{0},  i=1, \ldots, d-1\}$ are also divergence free vector fields on $\T^d$ defined as
\begin{equation}\label{xi}
    \xi_{k, i}(x) =\left\{\begin{array}{ll}
        2 \sqrt{C_{d} \kappa}\, \theta_{k} {\rm Re}(\sigma_{k, i}(x)), & \quad k \in \Z_{+}^d, \\
        2 \sqrt{C_{d} \kappa}\, \theta_{k} {\rm Im}(\sigma_{k, i}(x)), & \quad k \in \Z_{-}^d,
    \end{array} \right.
\end{equation}
where $\Z^d_0 = \Z^d_+ \cup \Z^d_-$ is a partition of $\Z^d_0$ satisfying $\Z^d_+ = -\Z^d_-$. $\{B^{k, i}: k\in\mathbb{Z}^{d}_{0},  i=1, \ldots, d-1\}$ are independent standard Brownian motions  defined as
\begin{equation}\label{B_t}
    B_t^{k, i}=\left\{\begin{array}{ll}
        {\rm Re}(W^{k, i}_{t}) ,& \quad k \in \Z_{+}^d, \\
        -{\rm Im}(W^{k, i}_{t}) ,& \quad k \in \Z_{-}^d.
    \end{array} \right.
\end{equation}

The SPDE \eqref{stoch-p-laplace} can be rewritten as the form \eqref{stoch-p-laplace.2} in the appendix.
By the above assumptions of noise, we have $\| \sigma_{k,i} \|_{L^\infty} \leq 1 , $ thus
\begin{equation*}
    \sum_{k \in \Z^d} \sum_{i=1}^{d-1} \| \xi_{k, i} \|_{L^{\frac{2p}{p-2}}(\T^d)}^{2} \leq 4C_d \kappa (d-1)\sum_{k \in \Z^d} \theta_k^2=4\kappa d <+\infty .
\end{equation*}
By Theorem \ref{well-posed}, we conclude that variational solutions of SPDE \eqref{stoch-p-laplace} exist and are unique; moreover, the energy identity \eqref{energy identity} holds.

\subsection{Mild formulation}\label{subsec:mild form}

Under the noise setting in Section \ref{subsec:noise setting}, we will show that the solution of SPDE \eqref{stoch-p-laplace} has a mild formulation, see \eqref{mild form.eq} below.
Transforming \eqref{stoch-p-laplace} into the It\^o form and similarly to the discussion in \cite[Section 2]{Galeati20}, we know that the Stratonovich-It\^o corrector is  $S(u)=\kappa \Delta u$, hence we obtain
\begin{equation}\label{p-Laplace.3}
    d u = \Delta_p u \, d t + \kappa \Delta u \, d t +
    \sum_{k \in \Z^d} \sum_{i=1}^{d-1} \xi_{k, i} (x) \cdot \nabla u \, d B_t^{k, i} ,
\end{equation}
where $\{\xi_{k, i}\}$ and $\{B_t^{k, i}\}$ are defined in \eqref{xi} and \eqref{B_t} respectively.

\begin{theorem}\label{mild form}
    Let $u$ be the variational solution of the SPDE \eqref{p-Laplace.3} in the sense of Definition \ref{variational solution}. Then $\P $-a.s. for any $t \in [0, T], $ it holds
    \begin{equation}\label{mild form.eq}
        u (t) = e^{\kappa t \Delta} u_0 + \int_0^t e^{\kappa (t - s) \Delta} \Delta_p u(s) \, \, d s
        +\sum_{k \in \Z^d} \sum_{i = 1}^{d - 1} \int_0^t e^{\kappa (t - s) \Delta} \xi_{k,  i} \cdot \nabla u(s) \, d B_s^{k, i} .
    \end{equation}
\end{theorem}

\begin{proof}
    For fixed $t \in [0, T], $ the variational solution $u$ of \eqref{stoch-p-laplace} satisfies
    \[ u (t) = u_0 + \int_0^t \Delta_p u \, \, d s +
    \int_0^t \kappa \Delta u \, \, d s + \sum_{k \in \Z^d} \sum_{i=1}^{d-1}
    \int_0^t \xi_{k, i} (x) \cdot \nabla u (s) \, \, d B_s^{k, i}  \]
    in $W^{- 1,  \frac{p}{p - 1}} (\T^d)$. By the property of Bochner integral, for any $\varphi \in C^{\infty} (\T^d), $ we have
    \begin{equation*}
        \begin{split}
            \langle u (t),  \varphi \rangle =&\, \langle u_0,  \varphi \rangle + \int_0^t \langle \Delta_p u(s),  \varphi \rangle \, \, d s
            + \kappa \int_0^t \langle \Delta u(s),  \varphi \rangle \, \, d s \\
            & + \sum_{k \in \Z^d} \sum_{i = 1}^{d - 1} \int_0^t \langle \xi_{k,  i} \cdot \nabla u(s), \varphi \rangle \, d B_s^{k, i} .
        \end{split}
    \end{equation*}
    For fixed $l \in \Z^d$, we take $\varphi=e_l$, then
    \begin{equation*}
        d \langle u (s),  e_l \rangle = \langle \Delta_p u(s),  e_l \rangle \, d s- \kappa \lambda_l \langle u(s),  e_l \rangle \, \, d s
        + \sum_{k \in \mathbb{Z}^d} \sum_{i = 1}^{d - 1} \langle \xi_{k,  i} \cdot \nabla u(s),  e_l \rangle \, d B_s^{k, i} .
    \end{equation*}
    Applying the It\^o formula to the process $e^{\kappa t \lambda_l}  \langle u(t) , \varphi \rangle ,$ we have
    \begin{equation*}
        \begin{split}
            d \big( e^{\kappa \lambda_l s} \langle u (s), e_l \rangle \big) & = \kappa \lambda_l e^{\kappa \lambda_l s}
            \langle u (s),  e_l \rangle \, \, d s + e^{\kappa \lambda_l s} d \langle u (s),  e_l \rangle \\
            & = e^{\kappa \lambda_l s} \langle \Delta_p u(s),  e_l \rangle \, \, d s + e^{\kappa \lambda_l s}
            \sum_{k \in \Z^d} \sum_{i = 1}^{d - 1}  \langle \xi_{k,  i} \cdot \nabla u(s),  e_l \rangle \, d B_s^{k, i} ,
        \end{split}
    \end{equation*}
    integrating in time yields, $\P$-a.s. for all $t \in [0, T]$,
    \begin{equation*}
        \begin{split}
            \langle u (t),  e_l \rangle &= e^{- \kappa \lambda_l t} \langle u_0, e_l \rangle +
            \int_0^t e^{- \kappa \lambda_l (t - s)} \langle \Delta_p u(s),  e_l \rangle \, \, d s \\
            &+\sum_{k \in Z^d} \sum_{i = 1}^{d - 1}  \int_0^t e^{-\kappa \lambda_l (t - s)}
            \langle \xi_{k,  i} \cdot \nabla u(s),  e_l \rangle \, d B_s^{k, i}  .
        \end{split}
    \end{equation*}
    We can then find $\Gamma\subset \Omega$ of full probability such that the above equality holds
    for all $t\in [0, T]$ and all $l\in \Z^2$. But this is exactly \eqref{mild form.eq} written in Fourier modes.
\end{proof}

\section{Averaged dissipation enhancement}\label{sec:dissipation.1}

This section is devoted to showing dissipation enhancement in the sense of expectation by choosing appropriate intensity
$\kappa$ and coefficients $\{ \theta_{k} \}_{k \in \Z_0^d}$ of the noise. We fix $\beta>\frac{d}{2}+1,\, t_0>0$, let $t_n \assign n t_0, \, n \in \Z_+$,
and assume $p \in \big(2,\frac{2d}{d-2} \big),$ initial data $\|u_0\|_{L^2} \leq R.$
Then we define two constants $C_1(\kappa,t_0,R)$ and $C_2(\theta,\mu)$ as below:
\begin{equation}\label{C_1,C_2}
    \begin{split}
        C_1(\kappa,t_0,R) & \assign \Bigg\{ \frac{6}{(4 \pi^{2} \kappa t_{0})^2} +
        6 \bigg(\frac{C_{0}}{\kappa^{\eta} t_{0}^{\eta-1/p}} \bigg)^{2} R^{2-4/p} \Bigg\}^{\frac{p-2}{2}} ,\\
        C_2(\theta,\mu) & \assign  \Big( 6 d \mu^{\frac{2}{p}\big(1-\frac{1}{\beta} \big)} \Lambda_{\beta} \|\theta\|_{\infty}^{\frac{2}{\beta}} \Big)^{\frac{p-2}{2}} ,
    \end{split}
\end{equation}
where $\eta \assign \frac{1}{2}+\frac{d}{2}(\frac{1}{2}-\frac{1}{p})<1,$ $\| \theta \|_{\infty} \assign \sup_{k \in \Z_0^d} |\theta_k| , $ and $\Lambda_{\beta} \assign \big(\sum_{l \in \Z^d} \lambda_l^{1-\beta} \big)^{\frac{1}{\beta}} < + \infty.$

To prove Theorem \ref{thm-average}, we first prove

\begin{proposition}\label{dissipation.1}
    For any $ p \in \big(2,\frac{2d}{d-2} \big),$ $\mu >0$, and initial data satisfying $\| u_0 \|_{L^2} \leq R $, we can find appropriate noise intensity $\kappa$ and coefficients $\{\theta_k\}_{k \in \Z_0^d}$, depending only on parameters $d, p, t_0, R, \mu$ and verifying
    \begin{equation}\label{condition.1}
        \max \{ C_1(\kappa,t_0,R),C_2(\theta,\mu) \} \leq \frac{1}{1+\mu (p-2) t_0 R^{p-2} } ,
    \end{equation}
    such that the solution $u(t)$ of SPDE \eqref{p-Laplace.3} satisfies
    \begin{equation}\label{dissipation.1.eq}
        \E \|u(t_{n+1}) \|_{L^2}^2 \leq \frac{\E \|u(t_n) \|_{L^2}^2}{\Big(1+\mu (p-2) t_0 \left(\E \|u(t_n) \|_{L^2}^2 \right)^{\frac{p-2}{2}} \!\Big)^\frac{2}{p-2} }, \quad \forall\, n \in \N.
    \end{equation}
\end{proposition}

For any fixed $n \in \N$, we will prove inequality \eqref{dissipation.1.eq} in the following two different cases:
\begin{itemize}
    \item Case 1:
    \begin{equation} \label{case1.assumption}
        \E \int_{t_{n}}^{t_{n+1}}\|\nabla u(s)\|_{L^p}^{p} \, d s \geq \mu t_{0} \big(\E\left\|u\left(t_{n}\right)\right\|_{L^2}^{2} \!\big)^{\frac{p}{2}} ;
    \end{equation}
    \item Case 2:
    \begin{equation} \label{case2.assumption}
        \E \int_{t_{n}}^{t_{n+1}}\|\nabla u(s)\|_{L^p}^{p} \, d s < \mu t_{0} \big(\E\left\|u\left(t_{n}\right)\right\|_{L^2}^{2}\! \big)^{\frac{p}{2}} .
    \end{equation}
\end{itemize}
In Case 1, inequality \eqref{dissipation.1.eq} follows immediately from the lower bound \eqref{case1.assumption} and the energy identity, cf. \eqref{energy identity.n} below. Case 2 is much more complicated; we will use the assumption \eqref{case2.assumption} to estimate the stochastic convolutional term, which is essential for estimating $\E \| u (t_{n+1}) \|_{L^2}^2$.

We first deal with Case 1. The solution $u(t)$ satisfies energy identity \eqref{energy identity}; as a result,
\begin{equation}\label{energy identity.n}
    \| u (t_{n+1}) \|_{L^2}^2 = \| u (t_n) \|_{L^2}^2 - 2\int_{t_n}^{t_{n+1}} \| \nabla u(s) \|_{L^p}^p \, d s .
\end{equation}
It implies that the estimate of $\| u (t_{n+1}) \|_{L^2}^2$ can be obtained by estimating $\int_{t_n}^{t_{n+1}} \| \nabla u(s) \|_{L^p}^p \, d s$,
and  the assumption \eqref{case1.assumption} gives a lower bound of its expectation. In this case, we can prove the inequality \eqref{dissipation.1.eq} without the assumptions in Proposition \ref{dissipation.1}.

\begin{proposition}\label{case.1}
    The inequality \eqref{dissipation.1.eq} holds in Case 1.
\end{proposition}
\begin{proof}
    Taking expectation on both sides of \eqref{energy identity.n} and using assumption \eqref{case1.assumption}, we get
    \begin{equation*}
        \begin{split}
            \E \| u (t_{n+1}) \|_{L^2}^2 &= \E \| u (t_n) \|_{L^2}^2 - 2 \E \int_{t_n}^{t_{n+1}} \| \nabla u(s) \|_{L^p}^p \, d s \\
            & \leq \E \| u (t_n) \|_{L^2}^2 - 2 \mu t_{0} \big(\E\left\|u\left(t_{n}\right)\right\|_{L^2}^{2}\! \big)^{\frac{p}{2}} .
        \end{split}
    \end{equation*}
    Applying Lemma \ref{P1} with $x=\mu (p-2) t_0 \big(\E\left\|u\left(t_{n}\right)\right\|_{L^2}^{2} \! \big)^{\frac{p-2}{2}},$ we obtain inequality \eqref{dissipation.1.eq}.
\end{proof}

In Case 2, we will follow the idea of proof of \cite[Theorem 1.9]{FGL21c} and estimate $\|u(t_{n+1})\|_{L^2}$ by using the mild formulation of $u(t)$ and the semigroup method,
where the inequality \eqref{case2.assumption} is necessary to estimate the stochastic convolution term.

\begin{proposition}\label{case.2}
    Under the assumptions in Proposition \ref{dissipation.1}, inequality \eqref{dissipation.1.eq} holds in Case 2.
\end{proposition}

Firstly, we give some notations that are frequently used in proving Proposition \ref{case.2}; for any $r \in [t_n, t_{n+1}]$, we define
\begin{equation}\label{mild solution's component}
    \begin{split}
        v_{1}(r) &:=e^{\kappa\left(r-t_{n}\right) \Delta} u\left(t_{n}\right)  , \\
        v_{2}(r) &:=\int_{t_{n}}^{r} e^{\kappa(r-\tau) \Delta} \Delta_p u(\tau) , \\
        v_{3}(r) &:=\sum_{k \in \Z^{d}} \sum_{i=1}^{d-1} \int_{t_{n}}^{r} e^{\kappa(r-\tau) \Delta}
        \xi_{k,  i} \cdot \nabla u(\tau) \,  d B_{\tau}^{k,  i} .
    \end{split}
\end{equation}
By Theorem \ref{mild form}, we know $u(r)=v_1(r)+v_2(r)+v_3(r)$ for $t\geq r_n$. Then by the decreasing property of $t \rightarrow \|u(t)\|_{L^2}$, it holds

\begin{equation}\label{case2.1}
    \| u(t_{n+1}) \|_{L^2} \leq \frac{1}{t_0} \int_{t_n}^{t_{n+1}} \|u(r)\|_{L^2} \, d r
    \leq \frac{1}{t_0} \int_{t_n}^{t_{n+1}} \big( \|v_1(r)\|_{L^2}+\|v_2(r)\|_{L^2}+\|v_3(r)\|_{L^2} \big) \, d r .
\end{equation}
The estimate of $\|v_1(r)\|_{L^2}$ is straightforward, thus we focus on the last two terms.
\begin{lemma}\label{v_2}
    If $p \in \big(2,\frac{2d}{d-2} \big)$, then it holds $\P$-a.s.
    \begin{equation}\label{v_2.eq}
        \frac{1}{t_0} \int_{t_n}^{t_{n+1}} \|v_2(r)\|_{L^2} \, d r \leq \frac{C_{0}}{\kappa^{\eta} t_{0}^{\eta-1/p}}
        R^{1-2/p} \left\|u\left(t_{n}\right)\right\|_{L^2} ,
    \end{equation}
    where $\eta =\frac{1}{2}+ \frac{d}{2}(\frac{1}{2}-\frac{1}{p})<1, $ and the constant $C_0$ depends only on parameters $p$ and $d$.
\end{lemma}

\begin{proof}
    For fixed $r \in [t_n, t_{n+1}], $ it holds $\P$-a.s.
    \begin{equation*}
        \begin{split}
            \left\|v_{2}(r)\right\|_{L^2} &=\sup _{\|\phi\|_{L^2}=1}\left\langle\int_{t_{n}}^{r} e^{\kappa(r-\tau) \Delta} \Delta_p u(\tau) \,  d \tau,  \phi\right\rangle \\
            & \leq \sup _{\|\phi\|_{L^2}=1} \int_{t_{n}}^{r} \big\langle e^{\kappa(r-\tau) \Delta }\Delta_p u(\tau),  \phi \big\rangle \,  d \tau \\
            & \leq \sup _{\|\phi\|_{L^2}=1} \int_{t_{n}}^{r} \| e^{\kappa(r-\tau) \Delta} \phi \|_{W^{1,  p}}\left\|\Delta_{p} u(\tau)\right\|_{W^{-1,  \frac{p}{p-1}}} \,  d \tau .
        \end{split}
    \end{equation*}
    Applying Lemma \ref{P3} and Proposition \ref{pro.p-laplace}, we have
    \begin{equation*}
        \begin{split}
            \left\|v_{2}(r)\right\|_{L^2} & \leq C \sup _{\|\phi\|_{L^2}=1} \int_{t_{n}}^{r}(\kappa(r-\tau))^{-\eta}
            \|\phi\|_{L^2} \left\|\Delta_{p} u(\tau)\right\|_{W^{-1,  \frac{p}{p-1}}} \,  d \tau \\
            & \leq C \int_{t_{n}}^{r}(\kappa(r-\tau))^{-\eta}\|\nabla u(\tau)\|_{L^p}^{p-1} \,  d \tau;
        \end{split}
    \end{equation*}
    integrating in $r\in [t_n, t_{n+1}]$ yields
    \begin{equation*}
        \begin{split}
            \frac{1}{t_{0}} \int_{t_{n}}^{t_{n+1}}\left\|v_{2}(r)\right\|_{L^2} \, d r & \leq \frac{C}{t_{0}} \int_{t_{n}}^{t_{n+1}} \!\! \int_{t_{n}}^{r}(\kappa(r-\tau))^{-\eta}\|\nabla u(\tau)\|_{L^p}^{p-1} \,  d \tau \, d r \\
            & =\frac{C}{t_{0}} \int_{t_{n}}^{t_{n+1}}\|\nabla u(\tau)\|_{L^p}^{p-1} \int_{\tau}^{t_{n+1}}(\kappa(r-\tau))^{-\eta} \, d r \,  d \tau \\
            & \leq \frac{C}{\kappa^{\eta} t_{0}^{\eta} (1-\eta)}  \int_{t_n}^{t_{n+1}} \| \nabla u(\tau) \|_{L^p}^{p-1} \,  d \tau .
        \end{split}
    \end{equation*}
    Let $C_0=\frac{C}{1-\eta}$, by H\"older's inequality and energy identity \eqref{energy identity.n}, we have
    \[ \frac{1}{t_{0}} \int_{t_{n}}^{t_{n+1}}\left\|v_{2}(r)\right\|_{L^2} \, d r \leq
    \frac{C_{0}}{\kappa^{\eta} t_{0}^{\eta-1/p}} \left\|u\left(t_{n}\right)\right\|_{L^2}^{ 2(p-1)/p} .\]
    Recalling that $\| u(t_n)\|_{L^2} \leq \|u_0\|_{L^2} \leq R, $ we get the inequality \eqref{v_2.eq}.
\end{proof}

Next we deal with the most difficult term $v_3$ for which we will need \eqref{case2.assumption}.

\begin{lemma}\label{v_3}
    Let $\frac{1}{\alpha}+ \frac{1}{\beta}=1 , \beta > \frac{d}{2} +1$; under the assumption \eqref{case2.assumption}, we have
    \begin{equation}\label{v_3.eq}
        \mathbb{E}\left(\frac{1}{t_0} \int_{t_{n}}^{t_{n+1}}\left\|v_{3}(r)\right\|_{L^2} d r\right)^{2} \leq d \mu^{\frac{2\alpha}{p}}
        \Lambda_{\beta} \|\theta\|_{\infty}^{\frac{2}{\beta}} \, \E \left\|u\left(t_{n}\right)\right\|_{L^2}^{2} ,
    \end{equation}
    where $\| \theta \|_{\infty} \assign \sup_{k \in \Z_0^d} |\theta_k| , $
    and $\Lambda_{\beta} \assign \big(\sum_{l \in \Z^d} \lambda_l^{1-\beta} \big)^{\frac{1}{\beta}} < + \infty.$
\end{lemma}

\begin{proof}
    By Jensen's inequality and It\^o's isometry, it holds
    \begin{equation*}
        \begin{split}
            \E\left(\frac{1}{t_{0}} \int_{t_{n}}^{t_{n+1}}\left\|v_{3}(r)\right\|_{L^2} d r\right)^{2} & \leq \frac{1}{t_{0}} \int_{t_{n}}^{t_{n+1}} \E\left\|v_{3}(r)\right\|_{L^2}^{2} \, d r \\
            &=\frac{1}{t_{0}} \E \int_{t_{n}}^{t_{n+1}} \sum_{k \in \Z^{d}} \sum_{i=1}^{d-1} \int_{t_{n}}^{r} \left\|e^{\kappa(r-\tau) \Delta} \xi_{k,  i} \cdot \nabla u(\tau)\right\|_{L^2}^{2} \,  d \tau \, d r \\
            &=\frac{1}{t_{0}} \E \sum_{i=1}^{d-1} \sum_{k, l \in \Z^{d}} \int_{t_{n}}^{t_{n+1}}  \!\! \int_{t_{n}}^{r} e^{-2 \kappa \lambda_{l}(r-\tau)} \left| \left\langle \xi_{k,  i} \cdot \nabla u(\tau),  e_{l}\right\rangle \right|^{2} \,  d \tau \, d r .
        \end{split}
    \end{equation*}
    By the H\"older inequality with exponents $\frac{1}{\alpha}+\frac{1}{\beta}=1$, we get
    \begin{equation}\label{v_3.1}
        \begin{split}
            \E & \left(\frac{1}{t_{0}} \int_{t_{n}}^{t_{n+1}}\left\|v_{3}(r)\right\|_{L^2} d r\right)^{2} \\
            & \leq \frac{1}{t_0} \E \sum_{i=1}^{d-1} \sum_{k, l \in \Z^{d}} \int_{t_{n}}^{t_{n+1}}  \!\! \int_{t_{n}}^{r} \left(e^{-2 \kappa \lambda_{l}(r-\tau)} \lambda_l \left| \left\langle a_{k,  i} e_k \cdot \nabla u(\tau),  e_{l}\right\rangle \right|^{2}\right)^{\frac{1}{\alpha}} \\
            & \hskip70pt \times \left(e^{-2 \kappa \lambda_{l}(r-\tau)} \lambda_l^{1-\beta} \left| \left\langle a_{k,  i} e_k \cdot \nabla u(\tau),  e_{l}\right\rangle \right|^{2}\right)^{\frac{1}{\beta}} \, d \tau \, d r \\
            & \leq\frac{1}{t_0} I_1^{\frac{1}{\alpha}} I_2^{\frac{1}{\beta}} ,
        \end{split}
    \end{equation}
    where $I_1$ and $I_2$ are defined as below:
    \begin{equation*}
        \begin{split}
            I_1 &\assign \E \sum_{k \in \Z^{d}} \sum_{i=1}^{d-1} \sum_{l \in \Z^{d}} \int_{t_{n}}^{t_{n+1}}  \!\! \int_{t_{n}}^{r} e^{-2 \kappa \lambda_{l}(r-\tau)}
            \lambda_l \left| \left\langle\xi_{k,  i} \cdot \nabla u(\tau),  e_{l}\right\rangle \right|^{2} \,  d \tau \, d r , \\
            I_2 &\assign \E \sum_{k \in \Z^{d}} \sum_{i=1}^{d-1} \sum_{l \in \Z^{d}} \int_{t_{n}}^{t_{n+1}}  \!\! \int_{t_{n}}^{r} e^{-2 \kappa \lambda_{l}(r-\tau)}
            \lambda_l^{1-\beta} \left| \left\langle\xi_{k,  i} \cdot \nabla u(\tau),  e_{l}\right\rangle \right|^{2} \,  d \tau \, d r .
        \end{split}
    \end{equation*}
    Recalling the definition of the vector fields $\{\xi_{k,i}\}$ and using the Fubini theorem, it holds
    \begin{equation*}
        \begin{split}
            I_1 & =\frac{2 \kappa d}{d-1} \E \sum_{k \in \Z^{d}} \theta_k^2 \sum_{i=1}^{d-1} \sum_{l \in \Z^{d}} \int_{t_{n}}^{t_{n+1}}  \!\! \int_{\tau}^{t_{n+1}} e^{-2 \kappa \lambda_{l}(r-\tau)}
            \lambda_l \left| \left\langle a_{k,  i} e_k \cdot \nabla u(\tau),  e_{l}\right\rangle \right|^{2} \, d r \,  d \tau , \\
            I_2 & =\frac{2 \kappa d}{d-1} \E \sum_{k \in \Z^{d}} \theta_k^2 \sum_{i=1}^{d-1} \sum_{l \in \Z^{d}} \int_{t_{n}}^{t_{n+1}}  \!\! \int_{\tau}^{t_{n+1}} e^{-2 \kappa \lambda_{l}(r-\tau)}
            \lambda_l^{1-\beta} \left| \left\langle a_{k,  i} e_k \cdot \nabla u(\tau),  e_{l}\right\rangle \right|^{2}  \, d r \,  d \tau .
        \end{split}
    \end{equation*}

    We estimate $I_1$ and $I_2$ separately. Firstly,
    \begin{equation*}
        \begin{split}
            I_1 & = \frac{2 \kappa d}{d-1} \E \sum_{k \in \Z^{d}} \theta_k^2  \sum_{i=1}^{d-1} \sum_{l \in \Z^{d}} \int_{t_{n}}^{t_{n+1}} \left| \left\langle a_{k,  i} e_k \cdot \nabla u(\tau),  e_{l}\right\rangle \right|^{2}
             \int_{\tau}^{t_{n+1}} e^{-2 \kappa \lambda_{l}(r-\tau)} \lambda_l \, d r \,  d \tau  \\
            & \leq \frac{d}{d-1} \E \sum_{k \in \Z^{d}} \theta_k^2  \sum_{i=1}^{d-1} \int_{t_{n}}^{t_{n+1}} \sum_{l \in \Z^{d}} \left| \left\langle a_{k,  i} e_k \cdot \nabla u(\tau),  e_{l}\right\rangle \right|^{2} \,  d \tau.
        \end{split}
    \end{equation*}
    By Parseval's identity, we get
    \[\sum_{l \in \Z^{d}} \left| \left\langle a_{k,  i} e_k \cdot \nabla u(\tau),  e_{l} \right\rangle \right|^{2} = \| a_{k,  i} e_k \cdot \nabla u(\tau) \|_{L^2}^2 \leq \| \nabla u(\tau) \|_{L^2}^2 \]
    for all $k \in \Z^d, i=1, 2, \dots, d-1$ and all $\tau \in [t_n, t_{n+1}].$ Recalling that $\sum_{k \in \Z_0^d} \theta_k^2 =1,$ so
    \begin{equation*}
        I_1 \leq d \, \E \int_{t_{n}}^{t_{n+1}} \| \nabla u(\tau) \|_{L^2}^2 \,  d \tau
        \leq \, d t_0^{\frac{p-2}{p}} \left( \E \int_{t_{n}}^{t_{n+1}} \|  \nabla u(\tau) \|_{L^p}^p \,  d \tau \right)^{\frac{2}{p}} .
    \end{equation*}
    By the assumption \eqref{case2.assumption}, it holds
    \begin{equation}\label{v_3.2}
        I_1 \leq d \mu^{\frac{2}{p}} t_0 \E \| u(t_n) \|_{L^2}^2 .
    \end{equation}

    Now we turn to estimate $I_2.$ Recalling that $\{a_{k,i} e_k\}$ are divergence free, by integration by parts, we have
    \begin{equation*}
        \left| \left\langle a_{k,  i} e_k \cdot \nabla u(\tau),  e_{l}\right\rangle \right|^{2}
        =\left| \left\langle e_k u(\tau),  a_{k,  i} \cdot \nabla e_{l}\right\rangle \right|^{2}
        = (2\pi)^2 (a_{k, i} \cdot l)^2 \left| \left\langle e_k u(\tau),  e_{l}\right\rangle \right|^{2}
    \end{equation*}
    for all $k \in \Z^d, i=1, 2, \dots, d-1$ and all $\tau \in [t_n, t_{n+1}].$  So
    \begin{equation*}
        \begin{split}
            I_2 & \leq \frac{2 \kappa d}{d-1} \E \sum_{k \in \Z^{d}} \theta_k^2 \sum_{i=1}^{d-1} \sum_{l \in \Z^{d}} \int_{t_{n}}^{t_{n+1}}  \!\! \int_{\tau}^{t_{n+1}} e^{-2 \kappa \lambda_{l}(r-\tau)}
            \lambda_l^{1-\beta} (2\pi)^2 (a_{k, i} \cdot l)^2 \left| \left\langle e_k u(\tau),  e_{l}\right\rangle \right|^{2}  \, d r \,  d \tau \\
            & \leq \frac{2 \kappa d}{d-1} \E \sum_{k \in \Z^{d}} \theta_k^2 \sum_{i=1}^{d-1} \sum_{l \in \Z^{d}} \int_{t_{n}}^{t_{n+1}} \left| \left\langle e_k u(\tau),  e_{l}\right\rangle \right|^{2}
            \int_{\tau}^{t_{n+1}} e^{-2 \kappa \lambda_{l}(r-\tau)} \lambda_l^{2-\beta}  \, d r \,  d \tau \\
            & \leq d \, \E \sum_{k \in \Z^{d}} \theta_k^2 \sum_{l \in \Z^{d}} \lambda_l^{1-\beta}  \int_{t_{n}}^{t_{n+1}} \left| \left\langle e_k u(\tau),  e_{l}\right\rangle \right|^{2} \,  d \tau .
        \end{split}
    \end{equation*}
    By the Bessel inequality, we have
    \begin{equation}\label{v_3.3}
        \begin{split}
            I_2 & \leq d \| \theta \|_{\infty}^2 \E \sum_{l \in \Z^{d}} \lambda_l^{1-\beta}  \int_{t_{n}}^{t_{n+1}} \sum_{k \in \Z^{d}} \left| \left\langle e_k u(\tau),  e_{l}\right\rangle \right|^{2} \,  d \tau \\
            & \leq d \| \theta \|_{\infty}^2 \E \sum_{l \in \Z^{d}} \lambda_l^{1-\beta}  \int_{t_{n}}^{t_{n+1}} \| u(\tau) e_l \|_{L^2}^2 \,  d \tau \\
            & \leq d \| \theta \|_{\infty}^2  \Lambda_{\beta}^{\beta} t_{0}\, \E \| u(t_{n}) \|_{L^2}^2 .
        \end{split}
    \end{equation}
    Combining the estimate \eqref{v_3.1}, \eqref{v_3.2}, and \eqref{v_3.3}, we arrive at the inequality \eqref{v_3.eq}.
\end{proof}

We have estimated the last two terms of the inequality \eqref{case2.1} respectively in Lemmas \ref{v_2} and \ref{v_3}; now we are ready to prove Proposition \ref{case.2}.

\begin{proof}[Proof of Proposition \ref{case.2}]
    By the property of heat semigroup, we have
    \begin{equation}\label{v_1}
        \begin{split}
            \frac{1}{t_{0}} \int_{t_{n}}^{t_{n+1}}\left\|v_{1}(r)\right\|_{L^2} \, d r & \leq \frac{1}{t_{0}} \int_{t_{n}}^{t_{n+1}}
            \exp \left\{-4 \pi^{2} \kappa\left(r-t_{n}\right)\right\}\left\|{u}\left(t_{n}\right)\right\|_{L^2} \, d r \\
            &=\frac{1-\exp \left\{-4 \pi^{2} \kappa t_{0}\right\}}{4 \pi^{2} \kappa t_{0}} \left\|u\left(t_{n}\right)\right\|_{L^2}
            \leq \frac{1}{4 \pi^2 \kappa t_{0}} \left\|u\left(t_{n}\right)\right\|_{L^2} .
        \end{split}
    \end{equation}
    By \eqref{case2.1} and the Cauchy inequality, $\| u(t_{n+1})\|_{L^2}^2$ is dominated by
    \[ 3 \left\{  \left( \frac{1}{t_0} \int_{t_n}^{t_{n+1}} \|v_1(r)\|_{L^2} \, d r \right)^2
            + \left( \frac{1}{t_0} \int_{t_n}^{t_{n+1}} \|v_2(r)\|_{L^2} \, d r \right)^2
            + \left( \frac{1}{t_0} \int_{t_n}^{t_{n+1}} \|v_3(r)\|_{L^2} \, d r \right)^2 \right\} . \]
    Taking expectation and applying Lemmas \ref{v_2} and \ref{v_3}, we get
    \begin{equation}\label{case2.2}
        \begin{split}
            \E \| u(t_{n+1}) \|_{L^2}^2 & \leq \frac{1}{2} \left( C_1(\kappa,t_0,R)^{\frac{2}{p-2}} + C_2(\theta,\mu)^{\frac{2}{p-2}} \right)  \E \| u({t_n}) \|_{L^2}^2 \\
            & \leq \max \{ C_1(\kappa,t_0,R),C_2(\theta,\mu) \}^{\frac{2}{p-2}} \E \| u({t_n}) \|_{L^2}^2 .
        \end{split}
    \end{equation}
    By the condition \eqref{condition.1} and recalling $\| u(t_n) \|_{L^2} \leq \| u_0 \|_{L^2}  \leq R ,$ we get
    \begin{equation*}\label{case2.3}
        \E \|u(t_{n+1}) \|_{L^2}^2  \leq \frac{\E \|u(t_n) \|_{L^2}^2}{\left(1+\mu (p-2) t_0 R^{p-2}\right)^\frac{2}{p-2} } \leq \frac{\E \|u(t_n) \|_{L^2}^2}{\Big(1+\mu (p-2) t_0 \left(\E \|u(t_n) \|_{L^2}^2 \right)^{\frac{p-2}{2}}\Big)^\frac{2}{p-2} } .
    \end{equation*}
    Therefore, inequality \eqref{dissipation.1.eq} holds also in Case 2.
\end{proof}

To prove Theorem \ref{thm-average}, we will choose appropriate $t_0$, $\mu',$ and use Lemma \ref{P2} to iterate over the estimates in Proposition \ref{dissipation.1}, finally we get the estimate on $\|u(t)\|_{L^2}$ for any $t \geq 1$.

\begin{proof}[Proof of Theorem \ref{thm-average}]
    Given any $\mu>0,$ we choose $0<t_0<1$ and $\mu'$ satisfying $\mu'(1-t_0)>\mu.$ For any $t \geq 1$, there exists $n \in \N$ such that $nt_0= t_n \leq t < t_{n+1} =(n+1)t_0 ,$ then
    \[\mu' \frac{t_n}{t} \geq \mu' \left(1-\frac{t_0}{t}\right) \geq \mu'(1-t_0) \geq \mu .\]
    Applying Proposition \ref{dissipation.1} with $t_0$ and $\mu'$, we obtain
    \[ \E \|u(t_m) \|_{L^2}^2 \leq \frac{\E \|u(t_{m-1}) \|_{L^2}^2}{\Big(1+ 2 \mu'(p-2) t_0 \left( \E \|u(t_{m-1}) \|_{L^2}^2 \right)^{\frac{p-2}{2}} \!\Big)^\frac{2}{p-2}} , \quad m=1,2,\dots,n .\]
    Applying Lemma \ref{P2} with $x=\E \|u(t_{n-2}) \|_{L^2}^2$, $y=\E \|u(t_{n-1}) \|_{L^2}^2$, $z=\E \|u(t_n) \|_{L^2}^2$, and $a=b=\mu'(p-2)t_0$, we have
    \[ \E \|u(t_n) \|_{L^2}^2 \leq \frac{\E \|u(t_{n-2}) \|_{L^2}^2}{\Big(1+ 2 \mu'(p-2) t_0 \left( \E \|u(t_{n-2}) \|_{L^2}^2 \right)^{\frac{p-2}{2}} \! \Big)^\frac{2}{p-2}} .\]
    Continuing in this way, we finally get
    \begin{equation*}
        \E \|u(t_n) \|_{L^2}^2 \leq \dots \leq \frac{\|u_0 \|_{L^2}^2}{\big(1+\mu' (p-2) n t_0 \|u_0 \|_{L^2}^{p-2} \big)^\frac{2}{p-2} } .
    \end{equation*}
    By decreasing property of $t\to \|u(t)\|_{L^2}$, we have
    \begin{equation*}
        \begin{split}
            \E \|u(t) \|_{L^2}^2 &\leq \E \|u(t_n) \|_{L^2}^2 \leq \frac{\|u_0 \|_{L^2}^2}{\big(1+\mu' \times \frac{t_n}{t} (p-2) t \|u_0 \|_{L^2}^{p-2} \big)^\frac{2}{p-2} } \\
            & \leq \frac{\|u_0 \|_{L^2}^2}{\big(1+\mu (p-2) t \|u_0 \|_{L^2}^{p-2} \big)^\frac{2}{p-2} } .
        \end{split}
    \end{equation*}
    So for any $t \geq 1$, the inequality \eqref{thm-average.eq} holds.
\end{proof}

From the previous proof and choice of the noise intensity $\kappa$ and coefficient $\{\theta_k\}_{k \in \Z_0^d}$ in Proposition \ref{dissipation.1},
we can easily conclude that Theorem \ref{thm-average} holds when the intensity $\kappa$ and coefficient $\{\theta_k\}_{k \in \Z_0^d}$ satisfy
\begin{equation}\label{thm-average.condition}
    \max \left\{ C_1(\kappa,t_0,R),C_2\left(\theta,\frac{\mu}{1-t_0}\right) \right\} \times \left(1+ \frac{\mu t_0}{1-t_0} (p-2) R^{p-2}\right) < 1
\end{equation}
for some $t_0 \in (0,1),$ where $C_1(\kappa,t_0,R)$ and $C_2(\theta,\mu)$ are defined in \eqref{C_1,C_2}.

\section{Almost sure dissipation enhancement}\label{sec:dissipation.2}

This section deals with the dissipation enhancement in $\P$-a.s. sense. For this purpose,
we set $t_0=1$  and choose time-dependent coefficients $\{\theta_{k,i}\}$ which are step functions of time $t$ as below:
\begin{equation}\label{noise.3}
    \theta_{k} (t) \assign  \theta_{k,n} , \quad  t\in [n,n+1) .
\end{equation}
For any fixed $n \in \N,$ the parameters $\theta^{(n)} \assign \{ \theta_{k,n} \}_{k \in \Z_0^d}$ satisfy the requirements in Section \ref{subsec:noise setting}. We consider the following SPDE:
\begin{equation}\label{p-Laplace.4}
    d u = \Delta_p u \, d t + \kappa \Delta u \, d t +
    \sum_{k \in \Z^d} \sum_{i=1}^{d-1} \xi_{k, i} (x,t) \cdot \nabla u \, d B_t^{k, i} ,
\end{equation}
where the vector fields $\{\xi_{k,i}(x,t)\}$ are defined as below:
\begin{equation}\label{xi.1}
    \xi_{k, i}(x,t) =\left\{\begin{array}{ll}
        2 \sqrt{C_{d} \kappa}\, \theta_{k}(t) {\rm Re}(\sigma_{k, i}(x)), & \quad k \in \Z_{+}^d, \\
        2 \sqrt{C_{d} \kappa}\, \theta_{k}(t) {\rm Im}(\sigma_{k, i}(x)), & \quad k \in \Z_{-}^d,
    \end{array} \right.
\end{equation}
with $\{\sigma_{k,i}\}$ being defined in \eqref{sigma}. Because $\theta_k(t)$ is a step function, for any $n \in \N$,
the variational solution $u(t)$ of SPDE \eqref{p-Laplace.4} exists and is unique in each interval $[n,n+1]$,
then we get the existence and uniqueness of solution $u(t)$  in $[0,+\infty).$
Similarly, the energy identity \eqref{energy identity} and mild form \eqref{mild form.eq} also hold.

\begin{proposition}\label{dissipation.2}
    For any $2 < p < \frac{2d}{d-2}, $ $\mu_0 >0$ and initial data satisfying $\| u_0 \|_{L^2} \leq R , $
    we can find appropriate intensity $\kappa$ and coefficients $\{\theta^{(n)}\}_{n \in \N}$ (only depending on parameters $d, p, R, \mu_{0}$) so that the solution $u(t)$ of SPDE \eqref{p-Laplace.4} satisfies
    \begin{equation}\label{dissipation.2.eq1}
        \E \|u(n) \|_{L^2}^2 \leq \frac{\E \|u(n-1) \|_{L^2}^2}{\Big(1+\mu_0 (p-2) 2^{n-1}\! \left(\E \|u(n-1) \|_{L^2}^2\right)^{\frac{p-2}{2}} \Big)^\frac{2}{p-2} } , \quad \forall\, n \in \Z_+ .
    \end{equation}
    Moreover, we can get the following dissipation estimate:
    \begin{equation}\label{dissipation.2.eq2}
        \E \|u(n) \|_{L^2}^2 \leq \frac{\|u_0 \|_{L^2}^2}{\big(1+\mu_0 (p-2) (2^n-1) \|u_0 \|_{L^2}^{p-2} \big)^\frac{2}{p-2} } , \quad \forall\, n \in \Z_+ .
    \end{equation}
    To be precise, the intensity $\kappa$ and coefficients $\{\theta^{(n)}\}_{n \in \N}$ satisfy
    \begin{equation}\label{condition.2}
        \begin{split}
            C_1(\kappa,1,R) & \leq \min \left\{ \frac{1}{1+\mu_0 (p-2) R^{p-2} },\frac{1}{3} \right\}, \\
            \sup_{n \in \N}C_2(\theta^{(n)},\mu_n) & \leq \min \left\{ \frac{1}{1+\mu_0 (p-2) R^{p-2} },\frac{1}{3} \right\} ,
        \end{split}
    \end{equation}
    where $\mu_n =2^{n} \mu_0 $ and $\theta^{(n)} \assign \{ \theta_{k,n} \}_{k \in \Z_0^d}.$
\end{proposition}

\begin{proof}
    We are going to prove Proposition \ref{dissipation.2} by induction. Because $\theta_{k}(t)$ is a step function, the estimates in Section \ref{sec:dissipation.1} also hold with $t_0=1,$ $ \mu=\mu_n$ and $\theta =\theta^{(n)}.$

    For $n=1$, by the estimate \eqref{dissipation.1.eq}, we have
    \begin{equation*}
        \E \| u(1) \|_{L^2}^2\leq \frac{\|u_0 \|_{L^2}^2}{ \big(1+\mu_0 (p-2) \|u_0 \|_{L^2}^{p-2} \big)^\frac{2}{p-2}} .
    \end{equation*}
    Thus, inequalities \eqref{dissipation.2.eq1} and \eqref{dissipation.2.eq2} hold for $n=1$.

    For $n=N \geq 2,$ we assume the inequality \eqref{dissipation.2.eq2} holds for $n=N-1,$ i.e.
    \begin{equation}\label{inductive assumption}
        \E \|u(N-1) \|_{L^2}^2 \leq \frac{\|u_0 \|_{L^2}^2}{\big(1+\mu_0 (p-2) (2^{N-1}-1) \|u_0 \|_{L^2}^{p-2} \big)^\frac{2}{p-2} }.
    \end{equation}
    If $\E \|u(N) \|_{L^2}^2 = 0$, then obviously the inequalities \eqref{dissipation.2.eq1} and \eqref{dissipation.2.eq2} hold with $n=N$, hence we assume $\E \|u(N) \|_{L^2}^2 > 0$. In this case, $\E \|u(N-1) \|_{L^2}^2 \geq \E \|u(N) \|_{L^2}^2 > 0$, thus we can rewrite \eqref{inductive assumption} as
    \begin{equation}\label{inductive assumption.1}
        \frac{1}{ \left( \E \| u(N-1) \|_{L^2}^2 \right)^{\frac{p-2}{2}} } \geq \frac{1}{ \| u_0 \|_{L^2}^{p-2} } +\mu_0 (p-2) (2^{N-1}-1).
    \end{equation}
    We only need to prove \eqref{dissipation.2.eq1} in Case 2 of the last section, with $\mu=\mu_{N-1}$ and $\theta=\theta^{(N-1)}$. By the estimate \eqref{case2.2}, we have
    \begin{equation*}
        \begin{split}
            \frac{1}{ \left( \E \| u(N) \|_{L^2}^2 \right)^{\frac{p-2}{2}} } & -\frac{1}{ \left( \E \| u(N-1) \|_{L^2}^2 \right)^{\frac{p-2}{2}} } \\
            \geq & \left(\frac{1}{ \max \{ C_1(\kappa,1,R),C_2(\theta^{(N-1)},\mu_{N-1}) \}}-1 \right) \frac{1}{ \left( \E \| u(N-1) \|_{L^2}^2 \right)^{\frac{p-2}{2}} } .
        \end{split}
    \end{equation*}
    Condition \eqref{condition.2} implies $\max \big\{ C_1(\kappa,1,R),C_2(\theta^{(N-1)},\mu_{N-1}) \big\} \leq \frac{1}{3}$, which, combined with inequality \eqref{inductive assumption.1}, gives us
    \begin{equation*}
        \begin{split}
            \frac{1}{ \left( \E \| u(N) \|_{L^2}^2 \right)^{\frac{p-2}{2}} } -\frac{1}{ \left( \E \| u(N-1) \|_{L^2}^2 \right)^{\frac{p-2}{2}} }
            & \geq  2 \bigg( \frac{1}{ \| u_0 \|_{L^2}^{p-2} } +\mu_0 (p-2) (2^{N-1}-1) \bigg) \\
            & \geq  \mu_0 (p-2) 2^{N-1} .
        \end{split}
    \end{equation*}
    It is equivalent to
    \[ \E \|u(N) \|_{L^2}^2 \leq \frac{\E \|u(N-1) \|_{L^2}^2}{\left(1+\mu_0 (p-2) 2^{N-1} (\E \|u(N-1) \|_{L^2}^2)^{\frac{p-2}{2}}\right)^\frac{2}{p-2}}, \]
    so inequality \eqref{dissipation.2.eq1} holds for $n=N.$ By inequality \eqref{inductive assumption} and Lemma \ref{P2}, we have
    \[\E \|u(N) \|_{L^2}^2 \leq \frac{\|u_0 \|_{L^2}^2}{\left(1+\mu_0 (p-2) (2^{N}-1) \|u_0 \|_{L^2}^{p-2}\right)^\frac{2}{p-2} } ,\]
    then inequality \eqref{dissipation.2.eq2} holds for $n=N.$ The induction holds.
\end{proof}

Now we will use Proposition \ref{dissipation.2} and Borel-Cantelli lemma to prove Theorem \ref{thm-pathwise}.

\begin{proof}[Proof of Theorem \ref{thm-pathwise}]
    Assume $u_0 \ne 0$; for any $n \geq 1,$ we define
    \begin{equation}\label{def.A_n}
            A_n  \assign \left\{ \omega \in \Omega : \sup_{t \in [n, n + 1]} \| u (t) \|_{L^2}^2 \geq \frac{\| u_0 \|^2_2}{\big( 1+ \mu_0  (p - 2) n \left\| u_0 \right\|_{L^2}^{p - 2} \! \big)^{\frac{2}{p - 2}}} >0 \right\} .
    \end{equation}
    Then by Chebyshev's inequality and Proposition \ref{dissipation.2}, it holds
    \begin{equation}\label{P A_n}
        \begin{split}
            \P(A_n) & \leq \frac{\big( 1 + \mu_0  (p - 2) n \left\| u_0 \right\|_{L^2}^{p - 2} \!\big)^{\frac{2}{p - 2}}}{\| u_0 \|^2_2} \E \bigg(\sup_{t \in [n,n+1]} \|u(t)\|_{L^2}^2 \bigg) \\
            & \leq \left(  \frac{1 + \mu_0 (p - 2) n \left\| u_0 \right\|_{L^2}^{p - 2} }{1+\mu_0 (p-2) (2^n-1) \|u_0 \|_{L^2}^{p-2}} \right)^{\frac{2}{p - 2}} .
        \end{split}
    \end{equation}
    So $\sum_{n \in \Z_{+}} \P(A_n) < + \infty ,$ and by Borel-Cantelli lemma, for $\P$-a.s. $\omega \in \Omega,$ there exists a big $N(\omega) >1$ such that for any $n>N(\omega),$ it holds
    \begin{equation*}
        \sup_{t \in [n, n + 1]} \| u (t) \|_{L^2}^2 \leq \frac{\| u_0 \|^2_2}{\big( 1 + \mu_0  (p - 2) n \left\| u_0 \right\|_{L^2}^{p - 2} \!\big)^{\frac{2}{p - 2}}} ;
    \end{equation*}
    If we set $\exp \{-\frac{1}{0_+}\}=0,$ then equivalently,
    \begin{equation}\label{dissipation3.eq.2}
        \sup_{t \in [n,n + 1]} \exp \left( -\frac{1}{\| u (t) \|_{L^2}^{p - 2}} \right) \leq e^{-\mu_0 (p-2)n} \exp \left( -\frac{1}{\| u_0 \|_{L^2}^{p - 2}} \right) .
    \end{equation}
    For $ 0 \leq n \leq N(\omega) ,$ by the decreasing property of $t \rightarrow \|u(t)\|_{L^2}$, we have
    \begin{equation*}
        \begin{split}
            \sup_{t \in [n,n + 1]} & \exp \left( -\frac{1}{\| u (t) \|_{L^2}^{p - 2}} \right) \leq \exp \left( -\frac{1}{\| u_0 \|_{L^2}^{p - 2}} \right) \\
            & \leq e^{\mu_0 (p-2) N(\omega)} e^{-\mu_0 (p-2) n} \exp \left( -\frac{1}{\| u_0 \|_{L^2}^{p - 2}} \right) .
        \end{split}
    \end{equation*}
    Thus, if we take $C(\omega) = e^{\mu_0 (p-2) (1+N(\omega))} ,$ it is easy to show
    \[\exp\left(- \frac1{\|u(t) \|_{L^2}^{p-2}} \right) \leq C(\omega) e^{-(p-2)\mu_0 t} \exp\left(- \frac1{\|u_0 \|_{L^2}^{p-2}} \right).\]
    And for any $t>\frac{\ln C(\omega)}{(p-2)\mu_0}$, we can get the decay of $\|u(t)\|_{L^2}$ as below:
    \[\|u(t) \|_{L^2} \leq \frac{\|u_0 \|_{L^2}}{\big(1+\left((p-2)\mu_0 t-\ln C(\omega) \right) \|u_0 \|_{L^2}^{p-2} \big)^{1/(p-2)}}.\]

    It remains to estimate the $q$-th moment of the random variable $C(\omega)$; to this end, we need to estimate the tail probability $\P(\{N(\omega) \geq k\})$.
    Note that $N(\omega)$ may be defined as the largest integer $n$ such that
      $$\sup_{t \in [n, n + 1]} \| u (t) \|_{L^2}^2 \geq \frac{\| u_0 \|^2_2}{\big( 1 + \mu_0  (p - 2) n \left\| u_0 \right\|_{L^2}^{p - 2} \! \big)^{\frac{2}{p - 2}}};$$
    hence
    \[\{\omega\in \Omega: N(\omega) \geq k\} = \bigcup_{n=k}^\infty A_n. \]
    Since $\|u_0\|_2>0,$ for any $\mu_0>0$, there exists $M=M(\mu_0, \|u_0\|_{L^2})$ such that
    \[ \frac{1 + \mu_0 (p - 2) n \left\| u_0 \right\|_{L^2}^{p - 2} }{1+\mu_0 (p-2) (2^n-1) \|u_0 \|_{L^2}^{p-2}} \leq \frac{2n}{2^n-1} \leq \bigg(\frac{1}{2} \bigg)^{\frac{n}{2}} ,\quad \forall\, n \geq M.\]
    By the estimate \eqref{P A_n}, for any $n \geq M,$ $\P(A_n) \leq \frac{1}{2^{\frac{n}{p-2}}},$ then for any $k \geq M$, it holds
    \[ \P(\{N(\omega) \geq k\}) \leq \sum_{n=k}^\infty \P(A_n) \leq \sum_{n=k}^\infty \left(\frac{1}{2}\right)^{\frac{n}{p-2}} = \frac{2^{\frac{1}{p-2}}}{2^{\frac{1}{p-2}}-1} \times \frac{1}{2^{\frac{k}{p-2}}} .\]
    As a result,
    \begin{equation}
        \begin{split}
            \E & e^{\mu_0 (p-2) q N(\omega)}  = \sum_{k=0}^\infty e^{\mu_0 (p-2) q k} \P(\{N(\omega) = k\}) \\
            & \leq \sum_{k=0}^{M} e^{\mu_0 (p-2) q k} + \frac{2^{\frac{1}{p-2}}}{2^{\frac{1}{p-2}}-1} \sum_{k=M+1}^\infty e^{ \big(\mu_0(p-2) q-\frac{\ln 2}{p-2} \big) k } .
        \end{split}
    \end{equation}
    When $q < \frac{\ln 2}{\mu_0 (p-2)^2} ,$ we get $ \E e^{\mu_0 (p-2) q N(\omega)} < +\infty $ and $C(\omega)$ has finite $q$-th moment.
\end{proof}

\begin{remark}\label{remark}
    In fact, in the same way, we can prove that for any positive coefficient polynomial $P(t)$, there exists a random constant $C(\omega)$ such that
    \begin{equation*}
        \exp \left(-\frac{1}{\| u(t) \|_{L^2}^{p-2}} \right) \leq C(\omega) e^{- \mu_0 (p-2) P(t)} \exp \left( -\frac{1}{\| u_0 \|_{L^2}^{p-2}} \right)
    \end{equation*}
    holds  $\P$-a.s. and for all $t \geq 0$.
\end{remark}

\appendix

\section{Variational solutions of stochastic $p$-Laplace equations with transport noise}\label{appendix}

In this part we consider general stochastic $p$-Laplace evolution equations of the form
\begin{equation}\label{stoch-p-laplace.1}
    d u = \Delta_p u \, d t + \sum_{n=0}^{+\infty} \xi_{n} \cdot \nabla u \circ d B_t^{n} ,
\end{equation}
where $\{B^n: n \in \N\}$ is a family of independent standard Brownian motions on some filtered probability space $(\Omega, \F, (\F_t), \P)$, and $\{\xi_{n}: n \in \N \}$ are some  divergence free vector
fields on $\T^d$ satisfying
\begin{equation}\label{condition}
    \eta \assign \sum_{n=0}^{+\infty} \| \xi_{k} \|_{L^{\frac{2p}{p-2}}(\T^d)}^{2} <+\infty .
\end{equation}
In It\^o form, the above equation \eqref{stoch-p-laplace.1} reads as
\begin{equation}\label{stoch-p-laplace.2}
    d u = \Delta_p u \, d t + S(u) \, d t + \sum_{n=0}^{+\infty} \xi_{n} \cdot \nabla u \, d B_t^{n} ,
\end{equation}
where the Stratonovich-It\^o correction term is now
\begin{equation}\label{Ito-Stratonovich}
    S(u)=\frac{1}{2} \sum_{n=0}^{+\infty} \xi_{n} \cdot \nabla \left( \xi_{n} \cdot \nabla u  \right) .
\end{equation}
We will prove the existence and uniqueness of its variational solution $u(t)$ and show that $u(t)$ satisfies the energy identity \eqref{energy identity}.

\begin{definition}\label{variational solution}
    A continuous $L^2(\T^d)$ valued $\mathcal{F}_t$-adapted process $u(t)$ is called a
    variational solution of \eqref{stoch-p-laplace.2}, if $u \in L^p (\Omega; L^p (0,  T ; W^{1, p})) \cap L^2 (\Omega ; C ([0,  T] ; L^2(\T^d))) $ and
    \begin{equation}\label{vraiational solution.eq}
        u (t) = u_0 + \int_0^t \Delta_p u \, \, d s +
        \int_0^t S(u) \, \, d s + \sum_{n=0}^{+\infty}
        \int_0^t \xi_{n} (x) \cdot \nabla u (s) \, d B_s^{n}
    \end{equation}
    holds in $W^{- 1,  \frac{p}{p - 1}} (\T^d)$ for any $t \in [0, T]$ and $\P$-a.s.
\end{definition}

Let $H=L^2(\T^d), V=W^{1, p}(\T^d), V^{\ast}=W^{-1, \frac{p}{p-1}}.$ Let $U$ be a separable Hilbert space,
and  $\left\{j_{n}, n \in \N \right\}$ be a complete orthonormal basis of $U$. $L_2 (U,  H)$ is the space of Hilbert-Schmidt operators from $U$ to $H$.
Define the operators $A: V \rightarrow V^{\ast}$ and $B: V \rightarrow L_2 (U,  H)$ as below:
\begin{equation}\label{operator}
    \begin{split}
        & A( u) \assign  \Delta_p u +\frac{1}{2} \sum_{n =0}^{+\infty}
        \xi_{n} \cdot \nabla \left( \xi_{n} \cdot \nabla u  \right)  , \quad \forall\, u \in V ,\\
        & B (u) (j_{n}) \assign  \xi_{n} \cdot \nabla u , \quad \forall\, u \in V, \forall\, n \in \N .
    \end{split}
\end{equation}

In order to prove the existence of variational solutions to \eqref{stoch-p-laplace.2}, we need to verify the conditions
in \cite[Theorem 4.2.4]{liu15}, listed below:
\begin{enumerate}
    \item[(H1)] (Hemicontinuity) For all $u, v, w \in V,$ the map
    \begin{equation*}
        \R \ni \lambda \to {}_{V^{\ast}}\langle A(u+\lambda v ),  w\rangle_V
    \end{equation*}
    is continuous.
    \item [(H2)] (Weak monotonicity) For all $u, v\in V,$
    \begin{equation*}
        2 {}_{V^{\ast}}\langle A( u)-A( v), u-v \rangle_V +\|B(u)-B(v)\|_{L_2(U, H)}^2 \leq 0 .
    \end{equation*}
    \item [(H3)] (Coercivity)  There exists $c_1>0$, such that for all $u \in V, t \in [0, T],$
    \begin{equation*}
        2 {}_{V^\ast}\langle A( u), u \rangle_V + \|B(u)\|_{L_2(U, V)}^2 \leq -c_1  \| u \|_V^2 .
    \end{equation*}
    \item [(H4)] (Boundedness) There exists a constant $c_2 >0$ such that for all $u \in V, t \in [0, T],$
    \begin{equation*}
        \|A(u) \|_{V^{\ast}} \leq c_2 \big(1+\| u\|_V^{p-1} \big).
    \end{equation*}
\end{enumerate}

\begin{theorem}\label{well-posed}
    If $u_0 \in L^2 \left( \Omega, \mathcal{F}_0, \P;L({\T^d}) \right) $and $\{\xi_n\}_{n \geq 1}$ satisfies \eqref{condition},
    then there exist a unique solution $\{ u(t) \}_{t \in [0,T]}$ to \eqref{stoch-p-laplace.2} in the sense of Definition \ref{variational solution}.
    Furthermore, the energy identity
    \begin{equation*}
        \| u (t) \|_{L^2}^2 = \| u_0 \|_{L^2}^2 - 2 \int_0^t \| \nabla u \|_{L^p}^p \, d s
    \end{equation*}
    holds $\P$-a.s. for all $t \geq 0$.
\end{theorem}

\begin{proof}
    Since $\xi$ is spatially divergence free, for all $ u, v \in V,$ we have
    \begin{equation}\label{operator A}
        {}_{V^{\ast}} \langle A (u),  v \rangle_V = - \int_{\T^d} |
        \nabla u |^{p - 2} \nabla u \cdot \nabla v\,  d x - \frac{1}{2} \sum_{n=0}^{+\infty}
        \int_{\T^d} (\xi_{n} \cdot \nabla u) (\xi_{n} \cdot \nabla v)\,  d x .
    \end{equation}
    By H\"older inequality,
    \begin{equation*}
       \begin{split}
        |{}_{V^{\ast}} \langle A (u),  v \rangle_V | & \leq \| \nabla u \|_{L^p}^{p - 1} \| \nabla v \|_{L^p}
        + \frac{1}{2} \| \nabla u \|_{L^p} \| \nabla v \|_{L^p}  \sum_{n=0}^{+ \infty}
        \left( \int_{\T^d}| \xi_{n} |^{\frac{2 p}{p - 2}}\,  d x \right)^{\frac{p - 2}{p}} \\
        & = \| \nabla u \|_{L^p}^{p - 1} \| \nabla v \|_{L^p} + \frac{\eta}{2} \| \nabla u \|_{L^p} \| \nabla v \|_{L^p} .
       \end{split}
    \end{equation*}
    As a result, the operator $A: V \rightarrow V^{\ast} $ is well-defined, and for any $u \in V$
    \[ \| A (u) \|_{V^{\ast}} \leq \| u \|_{W^{1,  p}}^{p - 1} + \eta \| u
     \|_{W^{1,  p}} \leq \eta + (1 + \eta) \| u \|_{W^{1,  p}}^{p - 1} ,\]
    so condition (H4) is satisfied.

    By the equality \eqref{operator A}, to prove that operator $A$ satisfies condition (H1), we only need to show for fixed $u, v, w \in V, $
    for $ \lambda \in \R,  |\lambda|<1,$
    \begin{equation*}
        \begin{split}
            \lim_{\lambda \rightarrow 0} & \int_{\T^d} \left( |\nabla (u+\lambda v)|^{p-2} \nabla (u+\lambda v) \cdot \nabla w
            -|\nabla u|^{p-2} \nabla u \cdot \nabla w \right) \, d x\\
            + & \frac{\lambda}{2} \sum_{n=0}^{+\infty} \int_{\T^d} \left( \xi_{n} \cdot \nabla v \right)
            (\xi_{n} \cdot \nabla w) \, d x =0 .
        \end{split}
    \end{equation*}
    By H\"older inequality and the condition \eqref{condition}, the second term tends to zero as $\lambda \rightarrow 0.$
    Since obviously, the integrands in the first term converge to zero as $\lambda \rightarrow 0, \,  d x$ -a.s., we only have to find a dominating
    function to apply Lebesgue's dominated convergence theorem. But

    \begin{equation*}
        |\nabla (u+\lambda v)|^{p-1} \cdot |\nabla w | \leq 2^{p-2} \left(|\nabla u|^{p-1}+ \nabla  v|^{p-1} \right) \, |\nabla w | ,
    \end{equation*}
    so the first term tends to zero as $\lambda \rightarrow 0.$ Condition (H1) is satisfied. Next
    \begin{equation}\label{operator B}
        \| B (u) \|_{L_2 (U,  H)}^2  : =  \sum_{n=0}^{+\infty} \int_{\T^d} (B (u) (j_{n}))^2\,  d x
        = \sum_{n=0}^{+\infty} \int_{\T^d} (\xi_{n} \cdot \nabla u)^2\,  d x , \quad \forall \, u \in V.
    \end{equation}
    By H\"older inequality and condition \eqref{condition},  the operator $B$ is well-defined.

    By \eqref{operator A}, \eqref{operator B}, and Proposition \ref{pro.p-laplace},  for fixed $u, v \in V, $ we have
    \begin{equation*}
        \begin{split}
            2{}_{V^{\ast}} \langle A (u) - A (v),  u - v \rangle_V + \| B (u) - B (v) \|_{L_2 (U,  H)}^2
            & = 2{}_{V^{\ast}} \langle \Delta_p u - \Delta_p v,  u - v \rangle_V \leq 0 ; \\
            2{}_{V^{\ast}} \langle A (v),  v \rangle_V + \| B (v) \|_{L_2 (U,  H)}^2
            &={}_{V^{\ast}} \langle \Delta_p v,  v \rangle_V \leq - c_2 \| v \|_{W^{1, p}}^p ,
        \end{split}
    \end{equation*}
    hence the conditions (H2) and (H3) are satisfied.

    The existence and uniqueness of \eqref{stoch-p-laplace.2} are consequence of \cite[Theorem 4.2.4]{liu15}. Because of $\E \|u(t)\|_{L^2}^2 \leq \E \|u_0\|_{L^2}^2 <+\infty,$
    applying \cite[Theorem 4.2.5]{liu15},  we know solution $u(t)$ is continuous in $L^{2}(\T^d), $
    and $\P$-a.s., for all $t \in [0, T]$, the energy identity
    \[ \| u (t) \|_{L^2}^2 = \| u_0 \|_{L^2}^2 - 2 \int_0^t \| \nabla u\|_{L^p}^p \, \, d s   \]
    holds. By the divergence free property of $\xi_{n}$,  the noise part vanishes in energy type computations. So the above energy estimate is similar to the deterministic system.
\end{proof}

\bigskip

\noindent \textbf{Acknowledgements.} The second author would like to thank the financial supports of the National Key R\&D Program of China (No. 2020YFA0712700), the National Natural Science Foundation of China (Nos. 11931004, 12090014), and the Youth Innovation Promotion Association, CAS (Y2021002).

\end{document}